\documentclass[12pt,reqno]{amsart}
\usepackage{amsmath}
\usepackage{amssymb}
\usepackage{amstext}
\usepackage{comment}
\usepackage{bm}
\usepackage{physics}
\usepackage{a4wide}
\usepackage{graphicx}
\usepackage[numbers,sort&compress]{natbib}
\allowdisplaybreaks \numberwithin{equation}{section}
\usepackage{color}
\usepackage{cases}

\numberwithin{equation}{section}

\newtheorem{theorem}{Theorem}[section]
\newtheorem{proposition}[theorem]{Proposition}

\newtheorem{lemma}[theorem]{Lemma}
\newtheorem*{theoremA}{Theorem A}
\theoremstyle{definition}

\theoremstyle{remark}
\newtheorem{remark}[theorem]{Remark}

\begin{document}

\title
[Quantitative stability of a class of explicit steady Euler flows in a disk]{Quantitative stability of a class of explicit steady Euler flows in a disk}

\author{Fatao Wang}
\address{School of Mathematical Sciences, Dalian University of Technology, Dalian 116024, PR China}
\email{wangfatao@mail.dlut.edu.cn}

\author{Guodong Wang}
\address{School of Mathematical Sciences, Dalian University of Technology, Dalian 116024, PR China}
\email{gdw@dlut.edu.cn}


\begin{abstract}
We provide a short proof of the $L^2$-orbital stability of a class of explicit steady Euler flows in a disk by establishing a quantitative estimate. The main idea is to exploit the conserved quantities of the Euler equation, including the kinetic energy, the enstrophy, and the moment of fluid impulse. Our result seems to suggest that more radial symmetry leads to stronger instability. 
\end{abstract}

\maketitle
\tableofcontents
\section{Introduction and main result}\label{sec1}

Let $D\subset\mathbb R^2$ be the unit disk centered at the origin.
Consider the evolution of an ideal fluid of unit density in $D$  with impermeability condition on $\partial D$, governed by the following two-dimensional Euler equation (vorticity form):
\begin{equation}\label{vor}
\begin{cases}
\partial_t\omega+\mathbf v\cdot\nabla\omega=0, & t\in\mathbb{R},\ \mathbf x=(x_1,x_2)\in D,\\
\mathbf v=\nabla^\perp\mathcal{G}\omega,\\
\omega(0,\cdot)=\omega_0,
\end{cases}
\end{equation}
where $\mathbf v(t,\mathbf x)$ is the velocity of the fluid, $\omega(t,\mathbf x)$ is the vorticity, $\nabla^\perp:=(\partial_{x_2},-\partial_{x_1}),$ and $\mathcal G$ is the inverse of $-\Delta$ subject to the zero Dirichlet boundary condition.  It is well known that the above Euler equation is globally well-posed for initial vorticity belonging to various function spaces; see \cite{BV,MB02,MP94}. In particular, if $\omega_0$ is smooth, then there exists a unique global smooth solution to \eqref{vor}.

For any smooth solution $\omega$ of the Euler equation \eqref{vor},
 the following three  conservation laws hold (cf. \cite{BV,MB02,MP94}):
 \begin{itemize}
 \item [(C1)] For any $t\in\mathbb R$,
 \[E(\omega(t,\cdot))=E(\omega(0,\cdot)),\]
 where 
     \begin{equation}\label{defofe}
     E(\omega ):=\frac{1}{2}\int_D\omega \mathcal G\omega \dd{\mathbf x}
     \end{equation}
 represents the \textit{kinetic energy} of the fluid;
 \item [(C2)] For any $s\in\mathbb R$ and any $t\in\mathbb R$, 
 \[\left|\left\{\mathbf x\in D\;\middle|\;\omega(t,\mathbf x)>s\right\}\right|=\left|\left\{\mathbf x\in D\;\middle|\;\omega(0,\mathbf x)>s\right\}\right|,\]
  where $|\cdot|$ denotes the two-dimensional Lebesgue measure;
    \item [(C3)] For any $t\in\mathbb R$,
 \[I(\omega(t,\cdot))=I(\omega(0,\cdot)),\]
 where 
      \begin{equation}\label{defofi}
      I(\omega):=\int_D\left|\mathbf x\right|^2\omega \dd{\mathbf x}
      \end{equation}
represents the \textit{moment of fluid impulse}.
 \end{itemize}
As a consequence of (C2),   the \emph{enstrophy} $J$ of the fluid, defined by
\begin{equation}\label{defofj}
J(\omega)=\int_D\omega^2 \dd{\mathbf x},
\end{equation}
is conserved under the Euler dynamics.
These three conservation laws play an essential role in this paper.

We will be focusing on a class of explicit steady solutions. Denote by $J_n$ the Bessel function of the first kind of order $n$, 
$$J_n(s):=\sum_{i=0}^{\infty}  \frac{(-1)^i}{i!(n+i)!}\left(\frac{s}{2}\right)^{n+2i},\quad s\in\mathbb R,\,n=0,1,2,\cdots,$$
and by $j_{n,k}$ the $k$-th positive zero of $J_n$.
For convenience, denote 
\[\mathsf j:=j_{1,1}\approx 3.831706.\]
Define
$$\mathbf V:=\operatorname{span}\left\{J_0(\mathsf jr),J_1(\mathsf jr)\cos\theta,J_1(\mathsf jr)\sin \theta \right\},$$
where $(r,\theta)$ denotes the polar coordinates in $\mathbb R^2$. 
According to \cite[Lemma 1.1]{W25},   any $\bar\omega\in\mathbf V$ is a steady solution to \eqref{vor}.

In \cite{W25}, the author showed that any steady flow with vorticity in $\mathbf V$  is orbitally stable up to rigid rotation. 
Throughout this paper, let $\mathbf O$ be a rotational orbit within $\mathbf V$, i.e.,
\begin{equation}\label{defor}
	\mathbf O:=\left\{AJ_0(\mathsf jr)+ BJ_1(\mathsf jr)\cos(\theta+\alpha) \;\middle|\;\alpha\in\mathbb R\right\}, 
\end{equation}
where $A,B\in\mathbb R$ are fixed. Without loss of generality, we may assume that $B\geq 0.$ For any solution $\omega(t,\mathbf x)$ to the Euler equation \eqref{vor}, we write $\omega_t:=\omega(t,\cdot)$ for simplicity. For $v\in L^p(D)$ and $K\subset L^p(D)$, we use $\mathrm{dist}_p(v,K)$ to denote the $L^p$-distance
between $v$ and $K$, i.e.,
\[\mathrm{dist}_p(v,K):=\inf_{w\in K}\|v-w\|_{L^p(D)}.\]
Then, the stability result in \cite{W25} can be stated as follows.
\begin{theoremA}[\cite{W25}]
Let $\mathbf{O}$ be given by \eqref{defor}, where $A, B \in \mathbb{R}$ and $B \geq 0$. Then, for any $1 < p < \infty$ and $\epsilon > 0$, there exists $\delta > 0$ such that, for any smooth solution $\omega(t, \mathbf{x})$ to \eqref{vor},  it holds that
\[
\mathrm{dist}_p(\omega_0, \mathbf{O}) < \delta 
\quad \Longrightarrow \quad 
\mathrm{dist}_p(\omega_t, \mathbf{O}) < \epsilon,
\ \forall\, t \in \mathbb{R}.
\]

\end{theoremA}

The main result of this paper is to improve Theorem A in the case $p = 2$ by specifying the dependence of $\delta$ on $\epsilon$, as stated precisely below. 

\begin{theorem}\label{thmmain}
Let $\mathbf{O}$ be given by \eqref{defor}, where $A, B \in \mathbb{R}$ and $B \geq 0$. Then:
\begin{itemize}
    \item[(i)] If $B \neq 0$, there exists a generic constant $C > 0$ such that, for any $\varepsilon > 0$ and any smooth solution $\omega(t, \mathbf{x})$ of the Euler equation \eqref{vor}, it holds that
    \begin{equation}\label{thmi}
    \mathrm{dist}_2(\omega_0, \mathbf{O}) \leq \varepsilon 
    \  \Longrightarrow  \ 
    \mathrm{dist}_2(\omega_t, \mathbf{O}) \leq C B^{-1} \left(A^2 + B^2\right)^{1/2} \varepsilon + C B^{-1} \varepsilon^2,
    \ \forall\, t \in \mathbb{R};
    \end{equation}
    
    \item[(ii)] If $B = 0$, there exists a generic constant $C > 0$ such that, for any $\varepsilon > 0$ and any smooth solution $\omega(t, \mathbf{x})$ of the Euler equation \eqref{vor}, it holds that
    \begin{equation}\label{thmii}
     \mathrm{dist}_2(\omega_0, \mathbf{O}) \leq \varepsilon 
    \quad \Longrightarrow \quad 
     \mathrm{dist}_2(\omega_t, \mathbf{O}) \leq C |A|^{1/2} \varepsilon^{1/2} + C \varepsilon,
    \  \forall\, t \in \mathbb{R}.
    \end{equation}
\end{itemize}
\end{theorem}
 
Unlike the qualitative result of Theorem A, our main result, Theorem \ref{thmmain}, is quantitative. The approach we use to prove Theorem \ref{thmmain} also differs from that of Theorem A (although both rely on the conserved quantities of the Euler equations). Specifically, Theorem A is essentially based on a compactness argument, whereas, to establish Theorem \ref{thmmain}, we must carry out a detailed quantitative analysis of the parameters that determine a rotational orbit, based on the conserved quantities. It should be emphasized that Theorem \ref{thmmain} concerns only the $L^2$ norm; obtaining a quantitative estimate for general $L^p$ norms appears to be difficult.

The idea of the proof is inspired by the works \cite{E24,WS99} on the orbital stability (up to translations) of the first Laplacian eigenstates on a two-dimensional flat torus. For the flat 2-torus $\mathbb T^2:=\mathbb R^2/(2\pi\mathbb Z^2),$ the first eigenspace of $-\Delta$ subject to mean-zero condition is   \[\operatorname{span}\left\{\cos x_1, \sin x_1, \cos x_2, \sin x_2\right\},\]
and a translational orbit within the first eigenspace can be written as 
\[ \left\{a\cos(x_1+\alpha)+b\cos(x_2+\beta)\;\middle|\; \alpha,\beta\in\mathbb R\right\},\]
where $a,b\in\mathbb R$. 
 To obtain orbital stability, Wirosoetisno and Shepherd \cite{WS99} first proposed using higher-order Casimirs to control the parameters $a$ and $b$. Motivated by \cite{WS99}, Elgindi \cite{E24} established the quantitative $L^2$-orbital stability  through several technical improvements (see also \cite{WZ23} for the qualitative $L^p$-orbital stability for any $1<p<\infty$).
Since each translation orbit is determined by two parameters, it suffices, in principle, to use only two conserved quantities to obtain the desired estimates. Indeed, Elgindi used the enstrophy and a modification of the quartic Casimir in his paper.
In the setting of this paper, any rotational orbit in $\mathbf{V}$ is also characterized by two parameters, so we still need to use two conserved quantities. Besides the enstrophy, we use the moment of fluid impulse $I$ in our proof. Compared to high-order Casimirs, the linear functional $I$ is simpler, making the associated calculations in the proof  considerably more straightforward.

Recently, the above method has been adapted by Delgadino and Melzi \cite{DM25} to establish the quantitative $L^2$-orbital stability for degree-2 Rossby-Haurwitz waves of the incompressible Euler equation on $\mathbb S^2$ (note that the qualitative $L^p$-orbital stability  for any $1<p<\infty$ was established in \cite{CWZ23}).  An interesting direction for future work is to apply this method to establish quantitative orbital stability for the first Laplacian eigenstate on a flat 2-torus of arbitrary shape, whose qualitative stability has been proved in \cite{W252}.

Finally, we emphasize that  Theorem \ref{thmmain} seems to suggest that \emph{more radial symmetry leads to stronger instability}.
To see this, we  assume  that $A^2+B^2=1$ and $\varepsilon\leq 1$. Then \eqref{thmi} becomes
    \[
    \mathrm{dist}_2(\omega_0, \mathbf{O}) \leq \varepsilon 
    \quad  \Longrightarrow  \quad 
    \mathrm{dist}_2(\omega_t, \mathbf{O}) \leq C B^{-1} \varepsilon 
    \ \ \forall\, t \in \mathbb{R}.
    \]
   Therefore, flow with more radial symmetry (larger 
$A$, smaller $B$) leads to a larger factor 
$B^{-1}$, which amplifies the perturbations.
 When the flow becomes purely radial,  \eqref{thmii} becomes
  \[\mathrm{dist}_2(\omega_0, \mathbf{O}) \leq \varepsilon 
    \quad \Longrightarrow \quad 
     \mathrm{dist}_2(\omega_t, \mathbf{O}) \leq C \varepsilon^{1/2} 
    \ \ \forall\, t \in \mathbb{R},
\]
which means that the stability exhibits a qualitative change shifting from $O(\varepsilon)$ to $O(\varepsilon^{1/2})$. See also Remark \ref{rk45}. Note that a similar phenomenon has also been observed in \cite{E24}: more symmetric flows (exact shear flows and exact cellular flows) appear to be less stable than less symmetric flows (flows with regions of both shearing and cat's-eye structures).

 This paper is organized as follows. In Section \ref{sec2}, we prove two Poincar\'e-type inequalities for later use. In Section \ref{sec3}, we establish the quantitative stability of the space $\mathbf V$. In Section \ref{sec4}, we give the proof of Theorem \ref{thmmain}.

\section{Two Poincar\'e-type inequalities}\label{sec2}
 
 We begin by considering the following Laplacian eigenvalue problem:
\begin{equation}\label{lep}
	\begin{cases}
		-\Delta u = \Lambda u, & \mathbf{x} \in D, \\
		u|_{\partial D} =c,\\
		\displaystyle \int_{D} u \, \dd{\mathbf{x}} = 0,
	\end{cases}
\end{equation} 
where $c$ is an unprescribed constant.
Denote by $\mathring L^2(D)$ the mean-zero subspace of $L^2(D)$, i.e.,
\[\mathring L^2(D):=\left\{v\in L^2(D) \;\middle|\;  \int_D v \dd{\mathbf x}=0\right\},\] 
 and define an operator $\mathcal T$ on $\mathring L^2(D)$ by setting
 \[\mathcal Tv:=\mathcal Gv-\frac{1}{|D|}\int_D\mathcal Gv \dd{\mathbf x},\quad v\in\mathring L^2(D).\]
 Then, it is easy to check that \eqref{lep} is equivalent to the following operator equation:
\[v=\Lambda \mathcal Tv,\quad v\in  \mathring L^2(D).\]
 Denote by $\Lambda_i$ the $i$-th eigenvalue (counted without multiplicity) of \eqref{lep}, and by $\mathbf E_i$ the corresponding eigenspace.  
Since $\mathcal T$ is a compact, symmetric, and positive-definite operator on $\mathring L^2(D)$, all eigenvalues can be arranged as
\begin{equation}\label{lmdai}
0<\Lambda_1<\Lambda_2<\dots\to+\infty,
\end{equation}
and all eigenspaces $\mathbf E_k$ are finite-dimensional and mutually orthogonal.

\begin{lemma}\label{lema21}
It holds that
	\begin{equation}\label{l1}
	\displaystyle \int_{D} v\mathcal{G}v \, \dd{\mathbf{x}} \leq \frac{1}{\Lambda_1}\displaystyle \int_{D} v^2 \, \dd{\mathbf{x}},\quad\forall\,v\in \mathring L^2(D), 
	\end{equation}
	and the equality holds if and only if $v\in\mathbf E_1;$
	and
 \begin{equation}\label{l2}
	\displaystyle \int_{D} v\mathcal{G}v \, \dd{\mathbf{x}} \leq \frac{1}{\Lambda_2}\displaystyle \int_{D} v^2 \, \dd{\mathbf{x}},\quad\forall\,v\in \mathbf E_1^{\perp},
	\end{equation}
		and the equality holds if and only if $v\in\mathbf E_2.$
\end{lemma}
\begin{proof}
Let $\mathcal P_i$ denote the orthogonal $L^2$-projector onto $\mathbf E_i$.  For any $v \in \mathring L^2(D)$, we can write  
$v=\sum_{i=1}^\infty\mathcal P_iv.$
 Then  
$\mathcal Tv=\sum_{i=1}^\infty \Lambda_i^{-1}\mathcal P_iv.$
By orthogonality, we have that
	\[
	\displaystyle \int_{D} v\mathcal{T}v \, \dd{\mathbf{x}} =\displaystyle \int_{D} \left(\sum_{i=1}^\infty\mathcal P_iv\right)\left(\sum_{j=1}^\infty\frac{1}{\Lambda_j}\mathcal P_jv\right) \, \dd{ \mathbf{x}} 
	=\sum_{i=1}^\infty\frac{1}{\Lambda_i}\displaystyle \int_{D} |\mathcal P_iv|^2 \, \dd{\mathbf{x}}.
	\]
On the other hand,
	\[
	\displaystyle \int_{D} v^2 \, \dd{\mathbf{x}} =\displaystyle \int_{D} \left(\sum_{i=1}^\infty\mathcal P_iv\right)\left(\sum_{j=1}^\infty \mathcal P_jv\right) \, \dd{\mathbf{x} }
	=\sum_{i=1}^\infty\displaystyle \int_{D} |\mathcal P_iv|^2 \, \dd {\mathbf{x}}.
	\]
In view of \eqref{lmdai}, it follows that 
	\begin{align*}
	\displaystyle \int_{D} v\mathcal{G}v \, \dd{\mathbf{x}} =\displaystyle \int_{D} v\mathcal{T}v \, \dd{\mathbf{x}}
	=\sum_{i=1}^\infty\frac{1}{\Lambda_i}\displaystyle \int_{D} |\mathcal P_iv|^2 \, \dd {\mathbf{x} }\leq 	\frac{1}{\Lambda_1}\sum_{i=1}^\infty\displaystyle \int_{D} |\mathcal P_iv|^2 \, \dd{\mathbf{x}}=\frac{1}{\Lambda_1}	\displaystyle \int_{D} v^2 \, \dd{\mathbf{x}},
	\end{align*} 
		and the equality holds if and only if $v\in\mathbf E_1.$
 This proves \eqref{l1}. To prove \eqref{l2}, notice that $\mathcal P_1 v\equiv 0$ for any  $v\in \mathbf  E_1^{\perp}$. Repeating the above calculations,
  we have that
	\begin{align*}
	\displaystyle \int_{D} v\mathcal{G}v \, \dd{\mathbf{x}} =\displaystyle \int_{D} v\mathcal{T}v \, \dd{\mathbf{x}}
	=\sum_{i=2}^\infty\frac{1}{\Lambda_i}\displaystyle \int_{D} |\mathcal P_iv|^2 \, \dd{\mathbf{x} }\leq 	\frac{1}{\Lambda_2}\sum_{i=2}^\infty\displaystyle \int_{D} |\mathcal P_iv|^2 \, \dd{\mathbf{x}}=\frac{1}{\Lambda_2}	\displaystyle \int_{D} v^2 \, \dd{\mathbf{x}},
	\end{align*}
		and the equality holds if and only if $v\in\mathbf E_2.$
 The proof is complete.
\end{proof}

The following lemma is crucial to subsequent discussion.
\begin{lemma}\label{lema22}
It holds that 
\[\Lambda_1=\mathsf j^2,\quad \mathbf E_1=\mathbf V.\]
\end{lemma}
\begin{proof}
See \cite[Proposition 3.4]{GL08} or \cite[Lemma 2.5]{W25}.
\end{proof}

\begin{remark}
  It is well known that the subspace $\mathrm{span}\left\{J_1(\mathsf jr)\cos\theta,J_1(\mathsf jr)\sin \theta \right\}$ of $\mathbf V$ is the  \emph{second} eigenspace of  \begin{equation*}
  \begin{cases} 
  -\Delta u=\lambda u,&\mathbf x\in D,\\u|_{\partial D}=0
  \end{cases}
  \end{equation*} with $\mathsf j^2$ being the second eigenvalue.
\end{remark}

\section{Quantitative stability of $\mathbf V$}\label{sec3}
A key step in the proof of Theorem \ref{thmmain} is the following quantitative estimate on the stability of the space $\mathbf{V}$ under mean-zero perturbations, which is obtained by applying the energy–Casimir method proposed by Arnold \cite{A65,A69}.

\begin{proposition}\label{prop31}
Let $\omega(t,\mathbf{x})$ be a smooth solution to the Euler equation~\eqref{vor}, satisfying $\omega_t = \omega(t,\cdot) \in \mathring{L}^2(D)$\footnote{This assumption is reasonable in view of the conservation law (C2) in Section~\ref{sec1}.}.
 Then
 \[\mathrm{dist}_2(\omega_t,\mathbf V)\leq  \sqrt{\frac{\Lambda_2}{\Lambda_2-\Lambda_1}} \mathrm{dist}_2(\omega_0,\mathbf V),\quad \forall\,t\in\mathbb R.\]
\end{proposition}
 
\begin{proof} 
Write
\[\omega_t=f_t+g_t,\quad f_t:=\mathcal P_1 \omega_t,\quad g_t:=\omega_t-\mathcal P_1\omega_t.\]
Then
 \[\mathrm{dist}_2(\omega_t,\mathbf V)=\|g_t\|_{L^2(D)}.\] According to Lemmas \ref{lema21} and \ref{lema22},
 \begin{equation}\label{fglqq1}
 \int_{D} f_t\mathcal{G}f_t  \dd{\mathbf{x}} = \frac{1}{\Lambda_1} \displaystyle \int_{D} f_t^2 \, \dd{\mathbf{x}},\quad   \int_{D} g_t\mathcal{G}g_t \, \dd{\mathbf{x}} \leq \frac{1}{\Lambda_2} \displaystyle \int_{D} g_t^2 \, \dd{\mathbf{x}}. 
 \end{equation}
 Define the energy-Casimir $EC$ by setting
 	\[EC(v):=2E(v)-\frac{1}{\Lambda_1}J(v)=\int_Dv\mathcal Gv \dd{\mathbf x}-\frac{1}{\Lambda_1}\int_Dv^2\dd {\mathbf x},\]
 which is a conserved quantity of the Euler equation (cf. \eqref{defofe} and \eqref{defofj}).
 	We calculate as follows:
\begin{align*}
	EC(\omega_t)&=\displaystyle \int_{D} (f_t+g_t)\mathcal{G}(f_t+g_t) \, \dd{\mathbf{x}}-
	\frac{1}{\Lambda_1}\displaystyle \int_D (f_t+g_t)^2 \dd{\mathbf x}\\
	&=\displaystyle \int_{D} f_t\mathcal{G}f_t \, \dd{\mathbf{x}}+\int_{D} g_t\mathcal{G}g_t \, \dd{\mathbf{x}}-\frac{1}{\Lambda_1}\left(\displaystyle \int_D f_t^2 \dd {\mathbf x}+\displaystyle \int_D g_t^2 \dd{\mathbf x}\right)\\
	&=\displaystyle \int_{D} g_t\mathcal{G}g_t \, \dd{\mathbf{x}}-\frac{1}{\Lambda_1} \displaystyle \int_D g_t^2 \dd{\mathbf x}.
\end{align*}
Using \eqref{fglqq1} and the fact that $\mathcal G$ is positive-definite, we have that
\[  -\frac{1}{\Lambda_1} \displaystyle \int_D g_t^2 \dd{\mathbf x}\leq EC(\omega_t) \leq \left(\frac{1}{\Lambda_2}-\frac{1}{\Lambda_1}\right)\int_D {g_t}^2 \dd{\mathbf x},\]
which holds for any $t\in\mathbb R$.
Then
\[-\frac{1}{\Lambda_1}\displaystyle \int_D {g_0}^2 \dd{\mathbf x}\leq EC(\omega_0)=EC(\omega_t) \leq \left(\frac{1}{\Lambda_2}-\frac{1}{\Lambda_1}\right)\displaystyle  \int_D {g_t}^2 \dd{\mathbf x},\]
which gives
\[\displaystyle \int_D {g_t}^2 \dd{\mathbf x} \leq \frac{\Lambda_2}{\Lambda_2-\Lambda_1} \displaystyle \int_D {g_0}^2 \dd{\mathbf x}.\]
This concludes the proof.
\end{proof}

\section{Proof of Theorem \ref{thmmain}}\label{sec4}
\subsection{Mean-zero perturbations}
First, we prove Theorem \ref{thmmain} for mean-zero perturbations. 
Consider a smooth solution   $\omega(t,\mathbf x)$ of the Euler equation \eqref{vor} satisfying $\omega_t=\omega(t,\cdot) \in\mathring L^2(D)$ for any $t\in\mathbb R$.
Denote
$ \hat\omega_t:=\mathcal P_1\omega_t\in\mathbf V.$ Then  \begin{equation}\label{perpv}
\omega_t-\hat\omega_t\in\mathbf V^\perp
\end{equation}
and
\begin{equation}\label{dhat}
\|\omega_t-\hat\omega_t\|_{L^2(D)}=\mathrm{dist}_2(\omega_t,\mathbf V).
\end{equation}
Since $\mathbf O$ is compact in $L^2(D)$,  
 there exists  $\tilde\omega_t\in\mathbf O$ such that
\begin{equation}\label{dtil}
\|\omega_t-\tilde\omega_t\|_{L^2(D)}=\mathrm{dist}_2(\omega_t,\mathbf O).
\end{equation} 
Our aim is to obtain a quantitative estimate of $\|\omega_t-\tilde\omega_t\|_{L^2(D)}$ in terms of $\|\omega_0-\tilde\omega_0\|_{L^2(D)}$.

From now on, we assume that
\[
    \|\omega_0-\tilde\omega_0\|_{L^2(D)}\leq \varepsilon,  
\]
where $\varepsilon>0$ (not necessarily small).
For convenience, we use $C$ to denote a \emph{generic} constant, whose value may change from line to line.

\begin{lemma}\label{lema401}
$\|\omega_t-\tilde\omega_t\|_{L^2(D)}\leq C \varepsilon+\|\hat\omega_t-\tilde\omega_t\|_{L^2(D)}.$
\end{lemma}
\begin{proof}
By \eqref{perpv} and Proposition \ref{prop31},
\begin{align*}
\|\omega_t-\tilde\omega_t\|^2_{L^2(D)}&=\|\omega_t-\hat\omega_t\|_{L^2(D)}^2+\|\hat\omega_t-\tilde\omega_t\|_{L^2(D)}^2\\
&\leq C\|\omega_0-\hat\omega_0\|_{L^2(D)}^2+\|\hat\omega_t-\tilde\omega_t\|_{L^2(D)}^2\\
&\leq C\|\omega_0-\tilde\omega_0\|_{L^2(D)}^2+\|\hat\omega_t-\tilde\omega_t\|_{L^2(D)}^2 \\
&\leq C\varepsilon^2+\|\hat\omega_t-\tilde\omega_t\|_{L^2(D)}^2.
\end{align*}
Then the desired estimate follows immediately.
\end{proof}

From Lemma \ref{lema401}, it suffices to estimate $\|\hat\omega_t-\tilde\omega_t\|_{L^2(D)}$.
 Since $\hat\omega_t,$ $\tilde\omega_t\in\mathbf V$, we can write 
\[\tilde\omega_t=A J_0(\mathsf jr)+B J_1(\mathsf jr)\cos(\theta+\alpha),\quad 
 \hat\omega_t=A'J_0(\mathsf jr)+B'J_1(\mathsf jr)\cos(\theta+\alpha'),\]
where $A, B, \alpha, A',B',\alpha'$ are all real numbers depending on $t$ and $B,B'\geq 0.$

\begin{lemma}\label{lema402}
$
\|\hat \omega_t-\tilde \omega_t\|_{L^2(D)}\leq C\left(|A-A'|+|B-B'|\right).
$
\end{lemma}
\begin{proof}
First, in view of \eqref{perpv}-\eqref{dtil}, it is easy to see that
\begin{equation}\label{e401}
\|\hat\omega_t-\tilde\omega_t\|_{L^2(D)}=\mathrm{dist}_2(\hat\omega_t,\mathbf O).
\end{equation}
On the other hand, using the fact that $J_0(\mathsf jr)$, $J_1(\mathsf jr)\cos \theta,$ and $J_1(\mathsf jr)\sin\theta$ are mutually orthogonal in $\mathring L^2(D)$, we can calculate as follows:
\begin{equation}\label{e402}
\begin{split}
    &\|\hat \omega_t-\tilde \omega_t\|^2_{L^2(D)}\\
	=&|A-A'|^2\|J_0(\mathsf jr)\|^2_{L^2(D)}+|B\cos \alpha-B'\cos\alpha'|^2\|J_1(\mathsf jr)\cos\theta\|^2_{L^2(D)}\\
	&+|B\sin  \alpha-B'\sin \alpha'|^2\|J_1(\mathsf jr)\sin\theta\|^2_{L^2(D)}\\
	=&|A-A'|^2\|J_0(\mathsf jr)\|^2_{L^2(D)}+\left(B^2+{B'}^2-2BB'\cos\left(\alpha-\alpha'\right)\right)\|J_1(\mathsf jr)\cos\theta\|_{L^2(D)}^2. 
\end{split}
\end{equation}
Combining \eqref{e401} and \eqref{e402}, we deduce that $BB'\cos(\alpha-\alpha')=BB'$. 
Then \eqref{e402} becomes
\begin{align*}
\|\hat \omega_t-\tilde \omega_t\|^2_{L^2(D)}=&|A-A'|^2\|J_0(\mathsf jr)\|^2_{L^2(D)}+|B-B'|^2\|J_1(\mathsf jr)\cos\theta\|_{L^2(D)}^2,
\end{align*}
which yields the desired result.
\end{proof}

According to Lemma \ref{lema401}, the problem reduces to estimating $|A-A'|$ and $|B-B'|$.
First, we  estimate $|A-A'|$ with the help of the moment of fluid impulse $I$ (cf. \eqref{defofi}).

\begin{lemma}\label{lema403}
$|A-A'|\leq C\varepsilon.$
\end{lemma}
\begin{proof}
First, we claim that 
\begin{equation}\label{claim1}
|I(\tilde\omega_t)-I(\hat\omega_t)|\leq C\varepsilon.
\end{equation}
In fact, since $I$ is conserved and rotationally invariant, it holds that
\[
I(\omega_t)=I(\omega_0),\quad I(\tilde\omega_t)=I(\tilde\omega_0).
\]
Then 
\begin{equation}\label{claim11}
\begin{split}
|I(\tilde\omega_t)-I(\hat\omega_t)|&= |I(\tilde\omega_0)-I(\hat\omega_t)|\\
&\leq  |I(\tilde\omega_0)-I(\omega_t)|+ |I(\omega_t)-I(\hat\omega_t)|\\
&=|I(\tilde\omega_0)-I(\omega_0)|+ |I(\omega_t)-I(\hat\omega_t)|.
\end{split}
\end{equation}
For the first term,   by H\"oler's inequality,
\begin{equation}\label{claim12}
|I(\tilde\omega_0)-I(\omega_0)|\leq \int_D|\tilde\omega_0 - \omega_0|\dd{\mathbf x} \leq C\|\tilde\omega_0 - \omega_0\|_{L^2(D)}\leq  C\varepsilon,
\end{equation}
For the second term, by \eqref{perpv} and Proposition \ref{prop31},
\begin{equation}\label{claim13}
|I(\omega_t)-I(\hat\omega_t)|\leq  C\|\omega_t - \hat\omega_t\|_{L^2(D)}\leq C\|\omega_0 - \hat\omega_0\|_{L^2(D)}\leq C\|\omega_0 - \tilde\omega_0\|_{L^2(D)}\leq C\varepsilon.
\end{equation}
The claim then follows from \eqref{claim11}-\eqref{claim13}. 
On the other hand, by a straightforward calculation,  
\begin{equation}\label{calc11}
	I(\tilde\omega_t)=\int_D|\mathbf x|^2(AJ_0(\mathsf jr)+BJ_1(\mathsf jr)\cos\left(\theta+\alpha\right)) \dd{\mathbf x}
=2\pi A\int_0^1r^3J_0(\mathsf jr) \dd {r}.
\end{equation}
We remark that  
 \[
\int_0^1 r^3J_0(\mathsf jr) \dd{r}\neq 0.
\]
In fact,
\[
 \int_0^1 r^3J_0(\mathsf jr) \dd{r} =\frac{1}{\mathsf j^4}\int_0^{\mathsf j}r^3J_0(r) \dd{r}
  =-\frac{2}{\mathsf j^4}\int_0^{\mathsf j} J_1(r)r^2\dd{r}
 <0,
\]
where we used   $(J_1(r)r)'=J_0(r)r$ (see \cite[p. 93]{Bo}) and the fact that $J_1$ is positive in $(0,\mathsf j)$.
Similarly, we have that
\begin{equation}\label{calc12}
I(\hat\omega_t)=2\pi A'\int_0^1r^3J_0(\mathsf jr) \dd {r}.
\end{equation}
From \eqref{calc11} and \eqref{calc12}, we get
\begin{equation}\label{aap}
I(\tilde\omega_t)-I(\hat\omega_t)
=2\pi \int_0^1r^3J_0(\mathsf jr) \dd {r}(A-A').
\end{equation}
Combining \eqref{claim1} and \eqref{aap}, we conclude the proof.
 
\end{proof}

Next, we estimate $|B - B'|$ using   the enstrophy $J$ (cf. \eqref{defofj}).

\begin{lemma}\label{lema404}
\begin{itemize}
    \item [(i)]If $B\neq 0,$ then 
\[|B-B'|\leq  CB^{-1}\left(A^2+B^2\right)^{1/2}\varepsilon+
CB^{-1}\varepsilon^2.\]
\item [(ii)]If $B\neq 0,$ then 
\[|B-B'|\leq C |A|^{1/2}\varepsilon^{1/2}+C\varepsilon.\]
\end{itemize}
 
\end{lemma}
\begin{proof}
First, we claim that 
\begin{equation}\label{claim2}
|J(\tilde\omega_t)-J(\hat\omega_t)|\leq C \left(A^2+B^2\right)^{1/2}\varepsilon+C\varepsilon^2.
\end{equation}
The proof is similar to that of \eqref{claim1}.
Since $J$ is also conserved and rotationally invariant, we have
\[
J(\omega_t)=J(\omega_0),\quad J(\tilde\omega_t)=J(\tilde\omega_0).
\]
Then, as in \eqref{claim11}, 
\begin{equation}\label{claim21}
|J(\tilde\omega_t)-J(\hat\omega_t)|\leq |J(\tilde\omega_0)-J(\omega_0)|+ |J(\omega_t)-J(\hat\omega_t)|.
\end{equation}
For the first term, 
\begin{equation}\label{claim22}
\begin{split}
|J(\tilde\omega_0)-J(\omega_0)|&\leq  \|\tilde\omega_0 + \omega_0\|_{L^2(D)} \|\tilde\omega_0-\omega_0\|_{L^2(D)}\\
& \leq \left(\|\tilde\omega_0 -\omega_0\|_{L^2(D)}+2\|\tilde\omega_0 \|_{L^2(D)}\right) \|\tilde\omega_0-\omega_0\|_{L^2(D)}\\
&\leq C \left(A^2+B^2\right)^{1/2}\varepsilon+C\varepsilon^2,
\end{split}
\end{equation}
where we used 
\[\|\tilde\omega_0\|_{L^2(D)}\leq C\left(A^2+B^2\right)^{1/2}.\]
For the second term,   
\begin{equation}\label{claim23}
\begin{split}
|J(\omega_t)-J(\hat\omega_t)|&\leq  \|\omega_t + \hat\omega_t\|_{L^2(D)} \| \omega_t-\hat\omega_t\|_{L^2(D)}\\
&\leq  2\|\omega_t\|_{L^2(D)}\|\omega_t - \hat\omega_t\|_{L^2(D)}\\
&= 2 \|\omega_0\|_{L^2(D)}\|\omega_t - \hat\omega_t\|_{L^2(D)}\\
&\leq 2 \left(\|\omega_0-\tilde\omega_0\|_{L^2(D)}+\|\tilde\omega_0\|_{L^2(D)}\right)\|\omega_t - \hat\omega_t\|_{L^2(D)}\\
&\leq C \left(A^2+B^2\right)^{1/2}\varepsilon+C\varepsilon^2.
\end{split}
\end{equation}
From \eqref{claim21}-\eqref{claim23}, we get \eqref{claim2}.

On the other hand,  
\begin{equation}\label{calc21}
\begin{split}
	J(\tilde\omega_t)&=\int_D(AJ_0(\mathsf jr)+BJ_1(\mathsf jr)\cos\left(\theta+\alpha\right))^2 \dd{\mathbf x}\\
	&= A^2\|J_0(\mathsf jr)\|^2_{L^2(D)}+B^2\|J_1(\mathsf jr)\cos\theta\|_{L^2(D)}^2.
\end{split}
\end{equation}
Similarly, we have that
\begin{equation}\label{calc22}
	J(\hat\omega_t)=A'^2\|J_0(\mathsf jr)\|^2_{L^2(D)}+B'^2\|J_1(\mathsf jr)\cos\theta\|_{L^2(D)}^2.
\end{equation}
From \eqref{calc21} and \eqref{calc22},  we get
 \begin{equation}\label{itih}
 J(\tilde\omega_t)-J(\hat\omega_t)
= \left(A^2-A'^2\right) {\|J_0(\mathsf jr)\|^2_{L^2(D)}}+\left(B^2-B'^2\right){\|J_1(\mathsf jr)\cos\theta\|_{L^2(D)}^2}. 
\end{equation}

Combining \eqref{claim2}, \eqref{itih}, and Lemma \ref{lema403}, we have that
\begin{equation}\label{eesm1}
\begin{split}
\left|B^2-B'^2\right|&\leq C \left(A^2+B^2\right)^{1/2}\varepsilon+C\varepsilon^2+C|A-A'||A+A'|\\
&\leq C \left(A^2+B^2\right)^{1/2}\varepsilon+C\varepsilon^2+C\varepsilon(|A-A'|+2|A|)\\
&\leq C \left(A^2+B^2\right)^{1/2}\varepsilon+C\varepsilon^2.
\end{split}
\end{equation}
In view of \eqref{eesm1}, we distinguish two cases:
\begin{itemize}
    \item [(1)] If $B\neq 0,$ then
    \[|B-B'|\leq \frac{C \left(A^2+B^2\right)^{1/2}\varepsilon+C\varepsilon^2}{B+B'}\leq \frac{C \left(A^2+B^2\right)^{1/2}\varepsilon+C\varepsilon^2}{B}.\]
    Note that we have assumed $B,B'\geq 0.$
    \item[(2)]If $B=0,$ then
    \[|B'|\leq C\left(|A|\varepsilon+C\varepsilon^2\right)^{1/2}\leq C |A|^{1/2}\varepsilon^{1/2}+C\varepsilon.\]
\end{itemize}
The proof is complete.
\end{proof}

\begin{remark}\label{rk45}
One may ask whether, in the case $B=0$, other Casimirs could be used to achieve an $O(\varepsilon)$ estimate for small $\varepsilon$. The answer is, in fact, negative.
In fact, for any   $f\in C^1(\mathbb R)$,
define
\[F(A,B):=\int_Df\left(AJ_0(\mathsf jr)+BJ_1(\mathsf jr)\cos\theta\right)\dd{\mathbf x}.\]
If we want to use $F$ to obtain an $O(\varepsilon)$ estimate by a similar argument, it is necessary to require $\partial_BF\neq 0.$
However, by a straightforward calculation,
\[\partial_BF|_{(A,0)}=\int_Df'\left(AJ_0(\mathsf jr) \right)J_1(\mathsf jr)\cos\theta \dd{\mathbf x}=0\]
for any $A\in\mathbb R.$
This indicates that radial flows (i.e., $B=0$) may indeed exhibit weaker stability.
\end{remark}

\begin{proof}[Proof of Theorem \ref{thmmain} for mean-zero perturbations]
It follows from \eqref{dtil} and Lemmas \ref{lema401}-\ref{lema404}.
 
\end{proof}

\subsection{General perturbations}
Fix a smooth solution $\omega_t$ to the Euler equation \eqref{vor} such that
\begin{equation}\label{lessvar}
\mathrm{dist}_2(\omega_0,\mathbf O)\leq \varepsilon.
\end{equation}
Define
\[w_t(\mathbf x):=\omega_t\left(\mathsf R_{\Omega t}\mathbf x\right)+2\Omega,\quad\Omega:=-\frac{1}{2|D|}\int_D\omega_0 \dd{\mathbf x},\]
where  $\mathsf R_{\Omega t}$ represents a clockwise rotation by an angle $\Omega t$.
Then, according to \cite[Lemma 6.1]{W25}, $w_t$ is also a solution to the Euler equation \eqref{vor}. Moreover,  it is clear that $w_t$ is mean-zero for any $t$.

\begin{lemma}\label{lema405}
$|\Omega|\leq C\varepsilon$.
\end{lemma}
\begin{proof}
Choose $v\in\mathbf O$ such that
\[\|\omega_0-v\|_{L^2(D)}=\mathrm{dist}_2(\omega_0,\mathbf O).\]
Then, by \eqref{lessvar}, 
\[|\Omega|=\frac{1}{2|D|}\left| \int_D\omega_0 \dd{\mathbf x}\right|= \frac{1}{2|D|}\left|\int_D\omega_0-v \dd{\mathbf x}\right|\leq C\|\omega_0-v\|_{L^2(D)}\leq C\varepsilon.\]
\end{proof}

\begin{lemma}\label{lema406}
It holds that
\[\mathrm{dist}_2(w_t,\mathbf O)-C\varepsilon\leq \mathrm{dist}_2(\omega_t,\mathbf O)\leq \mathrm{dist}_2(w_t,\mathbf O)+C\varepsilon.\]
\end{lemma}
\begin{proof}
From the definition of $w_t$ and Lemma \ref{lema405}, we have that
\[
\mathrm{dist}_2(w_t,\mathbf O) 
\leq \mathrm{dist}_2\left(\omega_t\circ \mathsf R_{\Omega t},\mathbf O\right)+\|2\Omega\|_{L^2(D)}
=\mathrm{dist}_2\left(\omega_t,\mathbf O\right)+\|2\Omega\|_{L^2(D)}
\leq \mathrm{dist}_2\left(\omega_t,\mathbf O\right)+C\varepsilon.
\]
The other inequality can be proved analogously.
\end{proof}

\begin{proof}[Proof of Theorem \ref{thmmain} for general perturbations]
Since  $w_t$ is a mean-zero solution to the Euler equation, we have, according to what we have proved in the last subsection, that 
\[\mathrm{dist}_2(w_0,\mathbf O)\leq \varepsilon\ \Longrightarrow\  \mathrm{dist}_2(w_t,\mathbf O)\leq
\begin{cases} CB^{-1}\left(A^2+B^2\right)^{1/2}\varepsilon+
CB^{-1}\varepsilon^2,\ \  &\mbox{if }B\neq 0,\\
C |A|^{1/2}\varepsilon^{1/2}+C\varepsilon,&\mbox{if }B= 0
\end{cases} \] 
for any $t\in\mathbb R.$ We can then conclude the proof by repeatedly applying Lemma \ref{lema406}.
\end{proof}

\bigskip
 \noindent{\bf Acknowledgements:}  G. Wang was supported by NNSF of China (Grant No. 12471101) and  Fundamental Research Funds for the Central Universities (Grant No. DUT19RC(3)075).

\bigskip
\noindent{\bf  Data Availability} Data sharing not applicable to this article as no datasets were generated or analysed during the current study.

\bigskip
\noindent{\bf Conflict of interest}    The authors declare that they have no conflict of interest to this work.

\phantom{s}
 \thispagestyle{empty}

\end{document}